\documentclass{birkjour}

\usepackage{amssymb}
\usepackage{amsmath}
\usepackage{amsthm}
\usepackage{tikz}
\usepackage{shortcuts}

\newtheorem*{theorem}{Theorem}

\begin{document}

\title{The conjugate locus on convex surfaces}

\author{Thomas Waters}

\address{T. Waters, Department of Mathematics, University of Portsmouth, England PO13HF, thomas.waters@port.ac.uk}

\maketitle

\begin{abstract}
The conjugate locus of a point on a surface is the envelope of geodesics emanating radially from that point. In this paper we show that the conjugate loci of generic points on convex surfaces satisfy a simple relationship between the rotation index and the number of cusps. As a consequence we prove the `vierspitzensatz':  the conjugate locus of a generic point on a convex surface must have at least four cusps. Along the way we prove certain results about evolutes in the plane and geodesic curvature. (Note: this is a corrected version of the original paper, see comment on page 5 and Appendix B).
\end{abstract}

\section{Introduction}

The conjugate locus is a classical topic in Differential Geometry (see Jacobi \cite{jacobi}, Poincar\'{e} \cite{poincare}, Blaschke \cite{blaschke} and Myers \cite{myers}) which is enjoying renewed interest due to the recent proof of the famous Last Geometric Statement of Jacobi by Itoh and Kiyohara \cite{Itoh1} (Figure \ref{cps}, see also \cite{Itoh3},\cite{Itoh2}, \cite{Sinclair1}) and the recent development of interesting applications (for example \cite{wolter1}, \cite{wolter2}, \cite{Bonnard3}, \cite{Bonnard5}). Recent work by the author \cite{TWbif} showed how, as the base point is moved in the surface, the conjugate locus may spontaneously create or annihilate pairs of cusps; in particular when two new cusps are created there is also a new ``loop'' of a particular type, and this suggests a relationship between the number of cusps and the number of loops. The current work will develop this theme, the main result being:

\setcounter{thm}{2}

\begin{thm}\label{thmrot}
Let $\mS$ be a smooth strictly convex surface and let $p$ be a generic point in $\mS$. Then the conjugate locus of $p$ satisfies \begin{equation} i=(n-2)/2, \label{iandnorig} \end{equation} where $i$ is the rotation index of the conjugate locus in $\mS/p$ and $n$ is the number of cusps.
\end{thm}

We will say more precisely what we mean by `generic' and `rotation index' in the next sections. It's worth pointing out that the main interest in the rotation index is that, for regular curves, $i$ is invariant under regular homotopy \cite{whitney}; however envolopes in general are not regular curves. Nonetheless the rotation index gives us a sense of the `topology' of the envelope, see for example Arnol'd \cite{arnold} (if the reader wishes we could perform a `smoothing' in the neighbourhood of the cusps to make the envelope regular, see \cite{lee}). This formula puts restrictions on the conjugate locus, and in particular we use it to then prove the so-called `vierspitzensatz':

\begin{thm}\label{thmvier}
Let $\mS$ be a smooth strictly convex surface and let $p$ be a generic point in $\mS$. The conjugate locus of $p$ must have at least 4 cusps.
\end{thm}

A central theme of this paper is the analogy between the conjugate locus on convex surfaces and the evolutes of simple convex plane curves; in Section 2 we focus on the latter and prove some results relating to rotation index. In Section 3 we focus on the conjugate locus and after some definitions we prove Theorems \ref{thmrot} and \ref{thmvier} and derive a formula for the geodesic curvature of the conjugate locus. Theorems 1-3 are new to the best of our knowledge, as is the expression for the geodesic curvature in Section 3. The proof of Theorem \ref{thmvier} is new and makes no reference to the cut locus, and the `count' defined in Section 3 is a development of some ideas in \cite{tabachfuchs}. Throughout a $'$ will mean derivative w.r.t. argument. 

\setcounter{thm}{0}

\section{Evolutes of plane convex curves}

Let $\bg$ be a `plane oval', by which we mean a simple smooth ($C^\infty$) strictly convex plane curve (we take the signed curvature $k_\gamma$ of $\bg$ to be negative for the sake of orientations, see below and Figure \ref{figevo}), and let $\bbet$ be the evolute to $\bg$ (the envelope of lines normal to $\bg$, also known as the caustic). The evolute has the following well known properties (see Bruce and Giblin \cite{bruce} or Tabachnikov \cite{tabachfuchs}): $\bbet$ will be a closed bounded curve (since $k_\gamma\neq 0$); $\bbet$ consists of an even number of smooth `arcs' separated by cusps which correspond to the vertices of $\bg$ (ordinary cusps if the vertices of $\bg$ are simple); and $\bbet$ is locally convex, by which we mean for every regular $q\in\bbet$ there is a neighbourhood of $\bbet$ which lies entirely to one side of the tangent line to $\bbet$ at $q$. 

\subsection{Rotation index}

We use the local convexity to define a `standard' orientation on $\bbet$ as follows: at any regular point we move along $\bbet$ such that the tangent line is to the left of $\bbet$ (changing the definition but keeping the sense of Whitney \cite{whitney}, see the black arrows in Figure \ref{figevo}). Note this implies $\bg$ is oriented in the opposite sense to $\bbet$. After performing one circuit of $\bbet$ the tangent vector will have turned anticlockwise through an integer multiple of $2\pi$; we call this the rotation index $i$ of $\bbet$ (at cusps the tangent vector rotates by $+\pi$, see DoCarmo \cite{docarmo}). Note: we use the term `rotation index' in keeping with \cite{docarmo}, \cite{kuhnel}, \cite{lee}, however some authors use `index' \cite{arnold}, `rotation number' \cite{whitney}, \cite{tabachfuchs}, `winding number' \cite{mccairns}, \cite{tabach} or `turning number' \cite{berg}.

\begin{figure}
\includegraphics[width=0.5\textwidth]{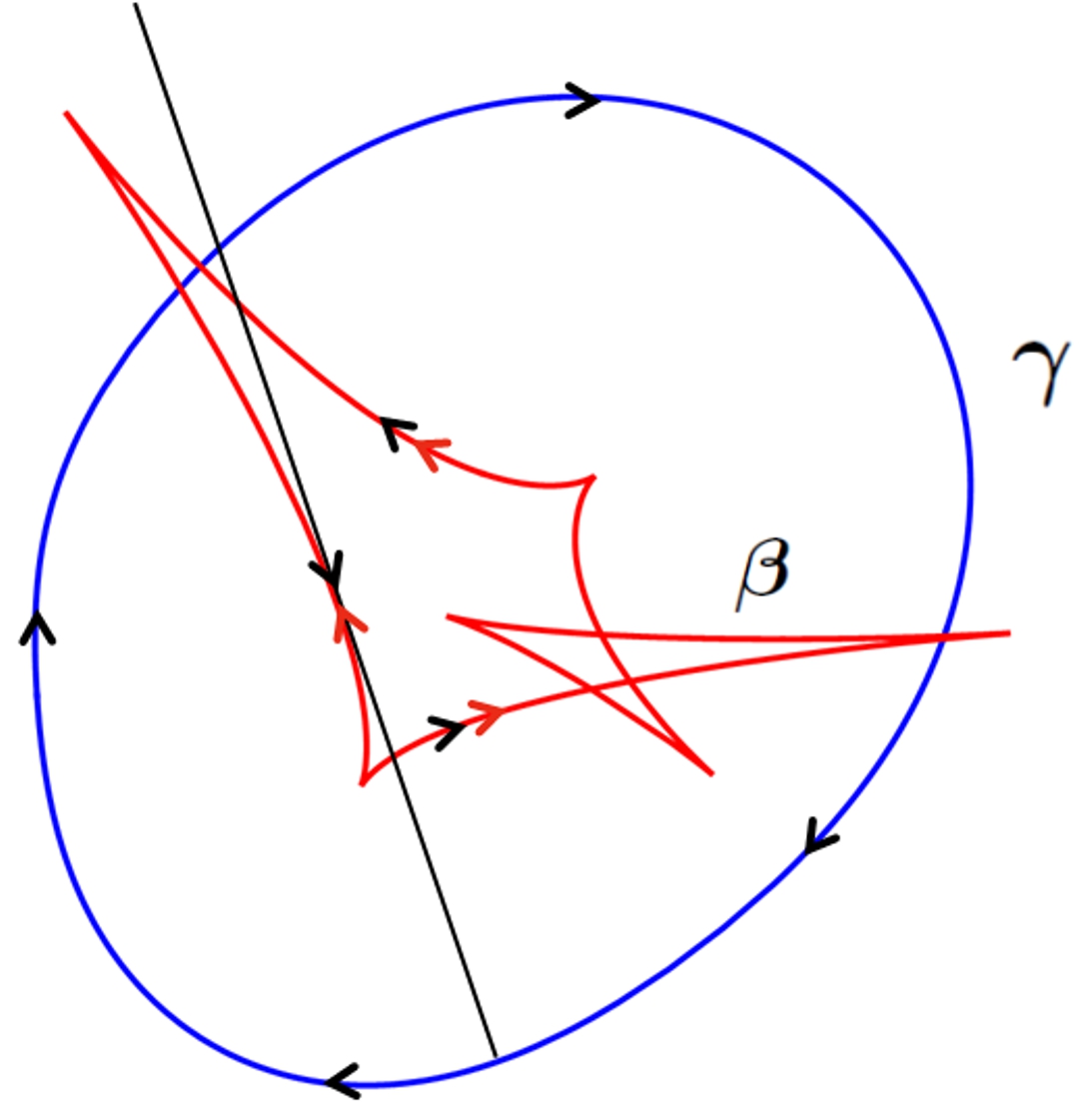}\caption{The orientations of $\bg$ (blue) and $\bbet$ (red) as defined in the text.}\label{figevo}
\end{figure}

\begin{lem}
A segment of $\bg$ between two vertices and the corresponding arc of $\bbet$ have total curvatures equal in magnitude but opposite in sign.
\end{lem}
\begin{proof}
By `total curvature' we mean the arc-length integral of the signed curvature. Let $s$ and $\sigma$ be the arc-length parameters of $\bg$ and $\bbet$ respectively. It is well known \cite{tabach} that if $R=R(s)$ is the radius of curvature of $\bg$ then the curvature of $\bbet$ is $k_\beta=-1/(RR')$ and $d\bbet/ds=R'\bN$ where $\bN$ is the unit normal to $\bg$. Therefore \[ \int_{\sigma_1}^{\sigma_2}k_\beta d\sigma= \int_{s_1}^{s_2}\left(\frac{-1}{RR'}\right)R' ds=-\int_{s_1}^{s_2}\frac{1}{R} ds=-\int_{s_1}^{s_2}k_\gamma ds. \]
\end{proof}

\begin{thm}
Let $\bg$ be a plane oval and $\bbet$ its evolute. The rotation index of $\bbet$, denoted $i$, is given by \[ i=(n-2)/2 \] where $n$ if the number of cusps of $\bbet$ (or vertices of $\bg$).
\end{thm}
\begin{proof}
The rotation index of $\bbet$ is defined by $2\pi i=\int d\phi$ where $\phi$ is the angle the unit tangent vector to $\bbet$ makes w.r.t. some fixed direction, thus \[ 2\pi i=\sum_{j=1}^n\int_{\sigma_j}^{\sigma_{j+1}}k_\beta d\sigma+n\pi \] where $j$ sums over the $n$ smooth arcs of $\bbet$ and the $n$ cusps of $\bbet$ are at $\sigma_j$. From Lemma 1 the first term on the right is just the total curvature of $\bg$ which is $-2\pi$ because $\bg$ is simple, smooth and oriented oppositely to $\bbet$. Therefore \be 2\pi i=-2\pi+n\pi\qquad \implies\qquad i=(n-2)/2. \label{iandn} \ee
\end{proof}

This simple formula puts restrictions on the form the evolute of a plane oval can take, for example if the evolute is simple it can only have 4 cusps (see Figure \ref{appics} for examples). For the benefit of the next Section we will rephrase this result in more general terms:

It is useful to imagine the tangent line to $\bbet$ rolling along the evolute in a smooth monotone way; this is due to the combination of local convexity of the smooth arcs and the ordinariness of the cusps. To make this more precise, let us define an `alternating' orientation on $\bbet$ by transporting the inward pointing normal to $\bg$ along the normal line to where it is tangent to $\bbet$ (red arrows in Figure \ref{figevo}), and let $\bt$ be the (unit) tangent vector to $\bbet$ w.r.t. this orientation. The oriented tangent line defined by $\bt$ makes an angle $\Phi$ w.r.t. some fixed direction, but importantly this angle varies in a smooth and monotone way as we move along $\bbet$ w.r.t. the standard orientation; in fact it changes by $-2\pi$ after one circuit of $\bbet$ (this is because $\Phi$ is simply the angle made by the normal to $\bg$, see Figure \ref{figevo}; this slightly overwrought description is necessary for later sections). 

The rotation index of $\bbet$ on the other hand is defined by the tangent vector w.r.t. the standard orientation. Let $\bT$ be this (unit) tangent vector and $\phi$ the angle it makes w.r.t. the same fixed direction. Since $\delta\phi=\pi$ at cusps then \[ 2\pi i=\sum_{j=1}^n(\delta\phi)_j+n\pi \] where $(\delta\phi)_j$ is the variation over the $j$th smooth arc. On the arcs where $\bT=\bt$ then $\delta\phi=\delta\Phi$, and on the arcs where $\bT=-\bt$ then $\phi=\Phi+\pi$ so again $\delta\phi=\delta\Phi$. Thus \[ 2\pi i=\sum_{j=1}^n(\delta\Phi)_j+n\pi=-2\pi+n\pi\quad \implies\quad i=(n-2)/2. \] In fact this is the case more generally, in the absence of any generating curve $\bg$ or local convexity (without $\bg$ some care is needed for point 2):

\begin{thm} \label{bnog}
Let $\bbet$ be a closed bounded piecewise-smooth planar curve with the following properties:\begin{enumerate}
\item $\bbet$ has an even number of smooth arcs separated by $n$ ordinary cusps, \item the alternating oriented tangent line completes a single clockwise rotation in traversing $\bbet$ w.r.t the standard orientation. \end{enumerate} Then the rotation index of $\bbet$ satisfies $i=(n-2)/2$. 
\end{thm}

It just happens that the evolute of a plane oval satisfies these conditions. Actually we can generalize to evolutes of plane curves with $k_\gamma\neq 0$ that are {\it not} simple: the term $\sum(\delta\phi)_j$ is just the rotation of the normal line to $\bg$, so if $I$ is the rotation index of $\bg$ then \be i=(n+2I)/2. \label{iandiandn} \ee As an example consider the well-known evolute of a non-simple lima\c{c}on, where $i=-1,I=-2,n=2$ (see Figure \ref{appics} in the Appendix).

\bigskip

\par\noindent\rule{\textwidth}{0.4pt}

\bigskip

{\bf Note to the reader:} In the original version of this paper there was a Section 2.2 titled `no smooth loops', which claimed that it was not possible to have a smooth arc of the evolute intersecting itself. Following communication with colleagues in the Universitat Aut\`{o}noma de Barcelona (March 2025) it was shown that these claims were false. Further details, including counterexamples to the `no smooth loops' claim, are in Appendix B, but please note this does not effect the validity of the other results in this Section and the next.

\bigskip

\par\noindent\rule{\textwidth}{0.4pt}

\bigskip

\section{Conjugate locus}

Let $\mS$ be a smooth strictly convex surface, and let $p$ be a point in $\mS$. The conjugate locus of $p$ in $\mS$, denoted $C_p$, is the envelope of geodesics emanating from $p$; put another way, $C_p$ is the set of all points conjugate to $p$ along geodesics emanating radially from $p$. We choose $\mS$ to be strictly convex so every geodesic emanating from $p$ reaches a point conjugate to $p$ in finite distance; indeed if $K$ is the Gauss curvature of $\mS$ then (using Sturm's comparison theorem, see \cite{blaschke},\cite{kling}) a conjugate point must be reached no sooner than $\pi/\sqrt{K_{max}}$ and no later than $\pi/\sqrt{K_{min}}$ (max/min w.r.t. $\mS$ as a whole). Following Myers \cite{myers}, we make the following definitions: fixing a direction in $T_p\mS$, the (unit) tangent vector of a radial geodesic at $p$ makes an angle $\psi$ w.r.t. this direction and this radial geodesic reaches a conjugate point after a distance $R$. We call $R=R(\psi)$ the `distance function' (another term being the `focal radius', see \cite{wolter1}), and the polar curve it defines in $T_p\mS$ the `distance curve' (see Figure \ref{figexp}). The distance curve lies in the annulus of radii $(\pi/\sqrt{K_{max}},\pi/\sqrt{K_{min}})$ and is a smooth closed curve (star-shaped w.r.t. $p$); it is the exponential map of this curve (i.e. the conjugate locus) that is cusped (see Figure \ref{figexp}).

The following facts about the conjugate locus of a point $p$ are well known (see \cite{docarmoriemm} or \cite{TWbif} for example): letting $\bGamma_\psi(s)$ be the unit speed geodesic with $\bGamma_\psi(0)=p$ and $\bGamma_\psi '(0)$ making an angle $\psi$ w.r.t. some fixed direction in $T_p\mS$, we define the exponential map as $X:T_p\mS\to\mS:X(s,\psi)=\bGamma_\psi(s)$. The Jacobi field $\bmt{J}=\partial X/\partial\psi$ satisfies the Jacobi equation, which we can write in scalar form as follows (in general everything depends on $s,\psi$ but for brevity we only denote the functional dependence where necessary): letting $(\bT,\bN)$ be the unit tangent and normal to $\bGamma_\psi(s)$ we write $\bmt{J}=\xi\bN+\eta\bT$ and we can take $\eta=0$ w.l.o.g. (see \cite{TWbif}); then $\xi$ satisfies the scalar Jacobi equation \be \pfrac{^2\xi}{s^2}+K\xi=0 \label{jaco} \ee where $K$ is evaluated along the radial geodesic and the partial derivatives are to emphasise that Jacobi fields vary along radial geodesics but also from one radial geodesic to another. With initial conditions $\xi(0,\psi)=0,\xi_{,s}(0,\psi)=1$ the first non-trivial zero of $\xi$ is a point conjugate to $p$. The distance curve is therefore defined implicilty by $\xi(R(\psi),\psi)=0$, and the conjugate locus in $\mS$ is $X(R(\psi),\psi)$. Letting $\bbet(\psi):\mathbb{S}^1\to\mS$ denote this parameterisation of $C_p$, then \be \bbet'=R'(\psi)\bT(R(\psi),\psi). \label{tangcp} \ee When $R$ is stationary then $\bbet$ is not regular, it has a cusp (see Figure \ref{figexp}).

\begin{figure}
\begin{center}
\includegraphics[width=0.8\textwidth]{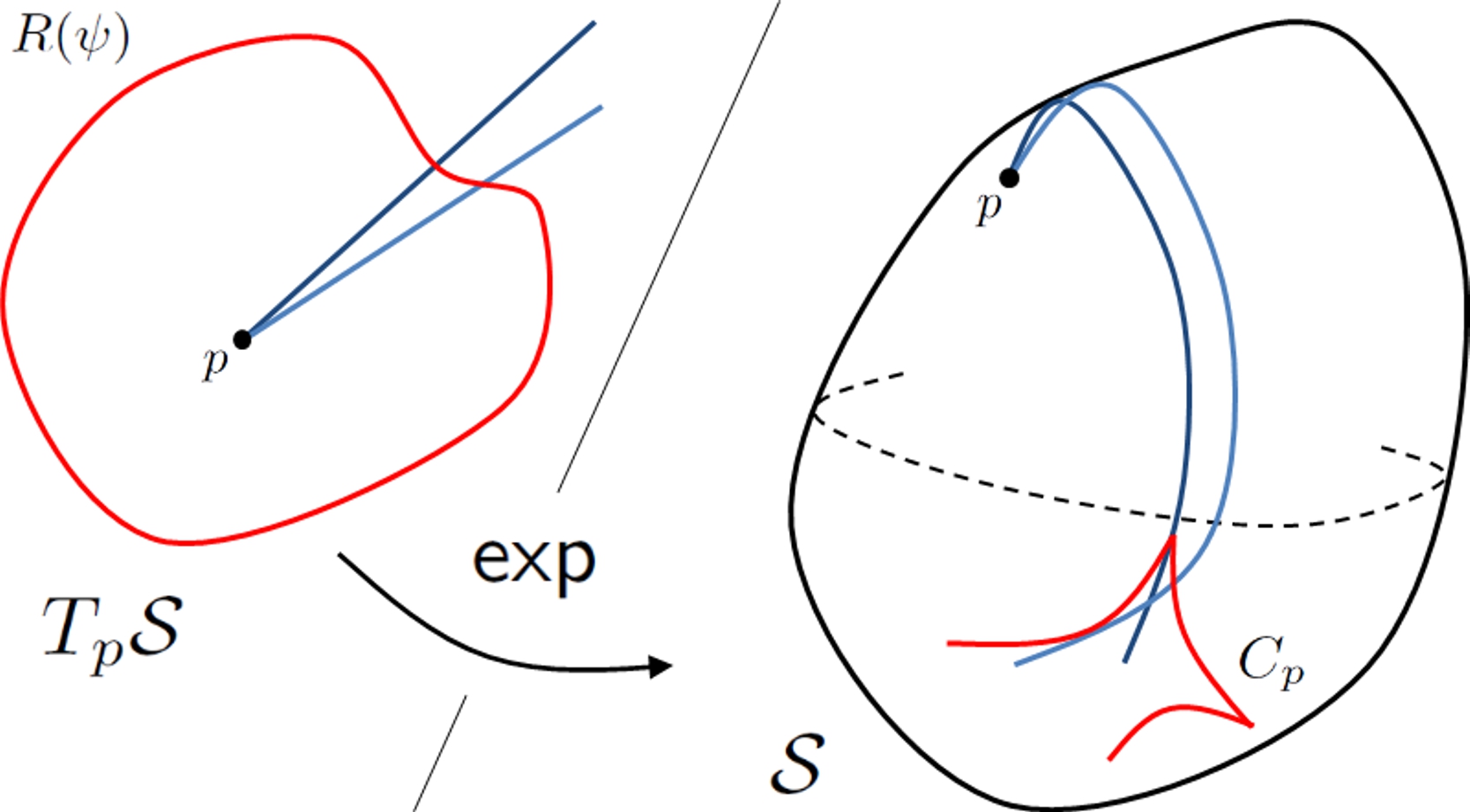}\caption{The distance curve in $T_p\mS$ and its image in $\mS$ under the exponential map: the conjugate locus of $p$, $C_p$.}\label{figexp}
\end{center}
\end{figure}

It is worth noting that the last expression tells us the following: $R$ is the arc-length parameter of $C_p$; the length of an arc of $C_p$ is simply the difference between consecutive stationary values of $R$ and since $R$ lives in an annulus this can bound the total length of $C_p$; if we give an alternating signed length to $C_p$ then its total length is zero; and finally the `string' construction of involutes carries over to conjugate loci (at $\psi_1$ move along a tangential geodesic a certain distance $\rho$, at $\psi_2$ move along a tangential geodesic a distance $\rho+R(\psi_2)-R(\psi_1)$; for some $\rho$ the involute will be a point, $p$).

\subsection{Rotation index}

In \cite{TWbif}, the author derived an expression for the local structure of the conjugate locus by considering the higher derivatives of the exponential map, and we reproduce the formula here: let $q$ be a point on the conjugate locus corresponding to $\psi=\psi_0$ and suppose $R$ has an $A_l$ singularity at $\psi_0$, by which we mean the first $l$ derivatives of $R$ vanish but not the $(l+1)$th. Then the leading behaviour in the parameterization of the conjugate locus at $\psi=\psi_0+\delta\psi$ is \be \bT\left[R^{(l+1)}\frac{\delta\psi^{l+1}}{(l+1)!}+\ldots\right]+\bN\left[(l+1)R^{(l+1)}\xi_{,s}\frac{\delta\psi^{l+2}}{(l+2)!}+\ldots\right]. \label{taylor} \ee We will return to the $\xi_{,s}$ term when we look at the geodesic curvature of $C_p$ (see \eqref{kgform}), for now we just need that $\xi_{,s}<0$. From this expression we see that if $R'\neq 0$ (i.e. an $A_0$ singularity of $R$) then locally $C_p$ is a parabola opening in the direction of $\bN$ if $R'<0$ and $-\bN$ if $R'>0$ (i.e. at regular points $C_p$ is locally convex), and if $R'=0$ (but $R''\neq 0$, an $A_1$ singularity) then $C_p$ has an ordinary (semicubical parabola) cusp. We will describe a point $p$ in $\mS$ as `generic' if the distance function $R$ has at most $A_1$ singularities.

It is well known that the rotation index of a curve on a surface can be defined by that of its preimage w.r.t. the coordinate patch of the surface which contains the curve, indeed Hopf's umlaufsatz does precisely that (see \cite{docarmo} or \cite{lee}). It is also well known that, with regards regular homotopy, every closed curve on a sphere has rotation index 0 or 1 \cite{mccairns}. For our purposes this is losing too much information about the curve, so (following Arnol'd \cite{arnold}) we define the rotation index of the conjugate locus $C_p$ as that on the surface $\mS$ with the point $p$ removed, which is equivalent to projecting the curve into a plane and finding the rotation index of that curve; we will essentially show this projected curve satisfies Theorem \ref{bnog}.

At the risk of confusing the reader we orient $\bbet$ in the following way: at $q\in\bbet$ we move in such a way that the radial geodesic from $p$ which has conjugate point at $q$ lies to the right of $\bbet$ w.r.t. the outward pointing normal of $\mS$ (this uses the local convexity of $\bbet$; alternatively $\psi$ increases in a clockwise sense when looking down on $p$, see Figure \ref{cps}); this is so the orientation of the projected curve agrees with the previous Section, and so $C_p$ on the ellipsoid has $i=1$. Also, if a radial geodesic from $p$ has conjugate point at $q$, we will refer to the portion $[p,q]$ as a `geodesic segment'.

\begin{thm}
Let $\mS$ be a smooth strictly convex surface and let $p$ be a generic point in $\mS$. Then the conjugate locus of $p$ satisfies $i=(n-2)/2$, where $i$ is the rotation index of the conjugate locus in $\mS/p$ and $n$ is the number of cusps.\end{thm}

\begin{proof} We already know the local structure of the conjugate locus of a generic point $p$ in $\mS$: smooth arcs separated by ordinary cusps. Let $\Pi$ be the projection defined as follows: radially project from $p$ to the unique supporting plane of $\mS$ parallel to the one at $p$. Since $\mS$ is smooth and convex this projection $\Pi:\mS/p\to\mathbb{R}^2$ is a homeomorphism. Letting $\ba=\Pi(\bbet)$, then $\ba$ is a piecewise-smooth curve with an even number of smooth arcs separated by ordinary cusps. If, on completing one circuit of $\ba$, the tangent line turns through $2\pi I$, then from Section 2 we have $i=(n+2I)/2$ where $i$ is the rotation index of $\ba$ in $\mathbb{R}^2$ (and hence the rotation index of $\bbet$ in $\mS/p$) and $n$ is the number of cusps to $\ba$ (and also $\bbet$). The only thing left to show is that $I=-1$. 

Let $C_\epsilon$ be a small geodesic circle in $\mS$ centred on $p$ with radius $\epsilon$. $\Pi(C_\epsilon)$ will be a simple curve in $\mathbb{R}^2$ which approaches a (large) circle as $\epsilon\to 0$. The projection under $\Pi$ of a geodesic segment gives a semi-infinite curve in $\mathbb{R}^2$ which meets $\Pi(C_\epsilon)$ at a unique point. For each $q\in\bbet$, we identify the tangent line to $\ba\in\mathbb{R}^2$ at $q'=\Pi(q)$ with this point of intersection of $\Pi(C_\epsilon)$ and the projected geodesic segment $[p,q]$. This identification is one-to-one and continuous, and since we traverse $\Pi(C_\epsilon)$ once in rotating about $p$ then the tangent line to $\ba$ makes one rotation and hence $I=\pm 1$. A simple diagram will show $I=-1$ with the orientations defined previously.

\end{proof}

An immediate corollary to this theorem would be an alternative route to proving the Last Geometric Statement of Jacobi: if we could show the conjugate locus of a generic point on the ellipsoid has rotation index 1 then there must be precisely 4 cusps.

\subsection{Counting segments}

In this section we will describe a way of assigning a positive intger to the regions of $\mS$ (developing ideas in \cite{tabach}) in order to give a convenient method of calculating the rotation index of $C_p$. We will use this, and the results of the previous section, to prove the `vierspitzensatz': that the conjugate locus of a point on a convex surface must have at least 4 cusps. In a footnote, Myers \cite{myers} refers to Blaschke \cite{blaschke} (in German) who in turn credits Carath\'{e}odory with a proof of this statement (see also p.206 of \cite{georg}).

We define on the complement of $\bbet$ in $\mS$ the following (which we shall refer to as the `count'): \[ m(o):\mS/\{\bbet\cup p\}\to\mathbb{N}/0 \] as the number of geodesic segments through the point $o\in\mS/\bbet$ (we also exclude $p$ as the count is infinte there). The count is constant on each component of $\mS/\bbet$, it is never equal to zero (see \cite{myers}), it is equal to 1 in the component of $\mS/\bbet$ containing $p$ (see Figures \ref{figeul} and \ref{cps} for examples) and it is preserved under the projection of the previous Subsection. As $o$ passes from one region of $\mS/\bbet$ (or $\mathbb{R}^2/\ba$) to another the count changes according to the rules shown in Figure \ref{figcount}; we know this from the locally convex structure of $\bbet$ described in \eqref{taylor}.

\begin{figure}
\begin{center}
\includegraphics[width=0.8\textwidth]{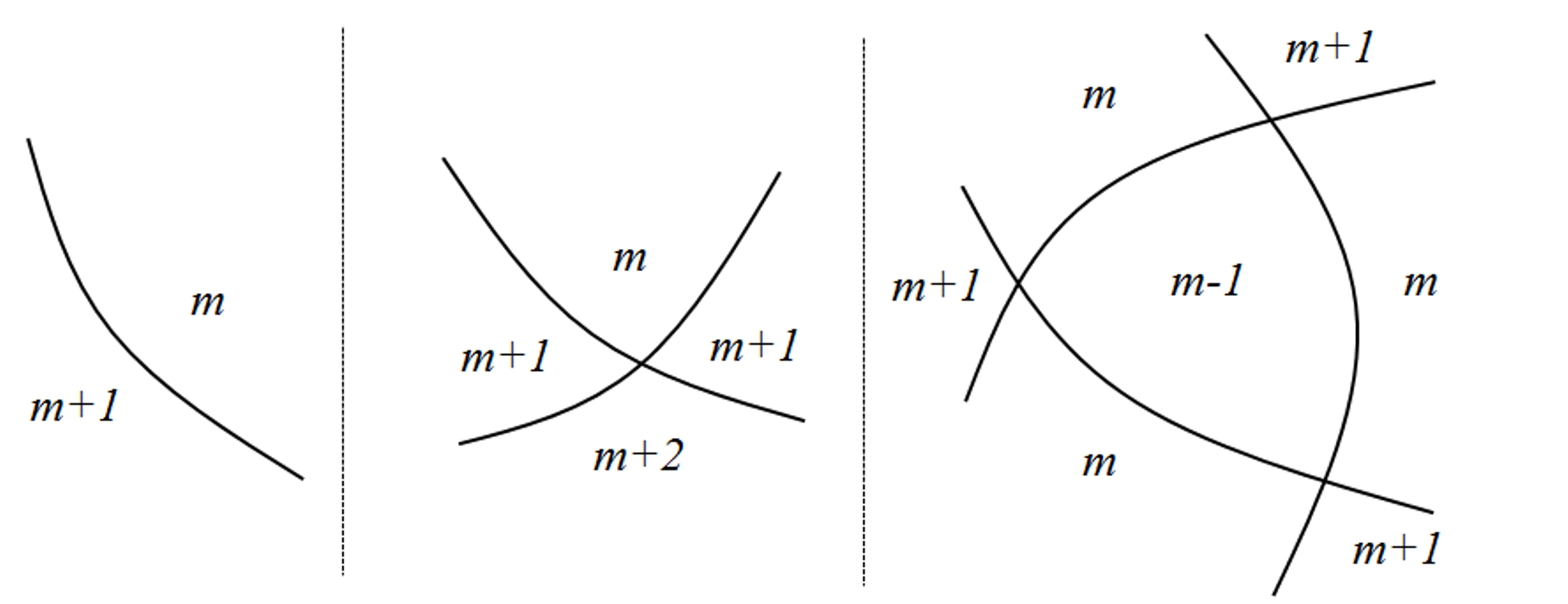}\caption{How the `count' changes as we pass from one component of $\mS/\bbet$ to another.}\label{figcount}
\end{center}
\end{figure}

We note another way of defining the count: the exponential map of the interior of the distance curve will cover $\mS$ and some points will be covered more than once; the number of times a point $o\in\mS/\bbet$ is covered by this map is $m(o)$. This alternative formulation may be useful in Gauss-Bonnet type constructions \cite{rogen}.

We can use the count to describe a method of calculating the rotation index of $\bbet$ due to McIntyre and Cairns \cite{mccairns}. In their paper they prove the following (we have adapted their statement to fit the current situation):

\begin{theorem} [McIntyre and Cairns] Assigning a number to each region of $\mS/\bbet$ according to the count, let $S_m$ be the union of the closure of the regions numbered $\geq m$, and $\chi_m$ the Euler characteristic of $S_m$. Then the rotation index of $\bbet$ in $\mS/p$ is \[ i= \sum_{m\geq 2} \chi_m.  \]
\end{theorem}

\begin{figure}
{\includegraphics[width=0.25\textwidth]{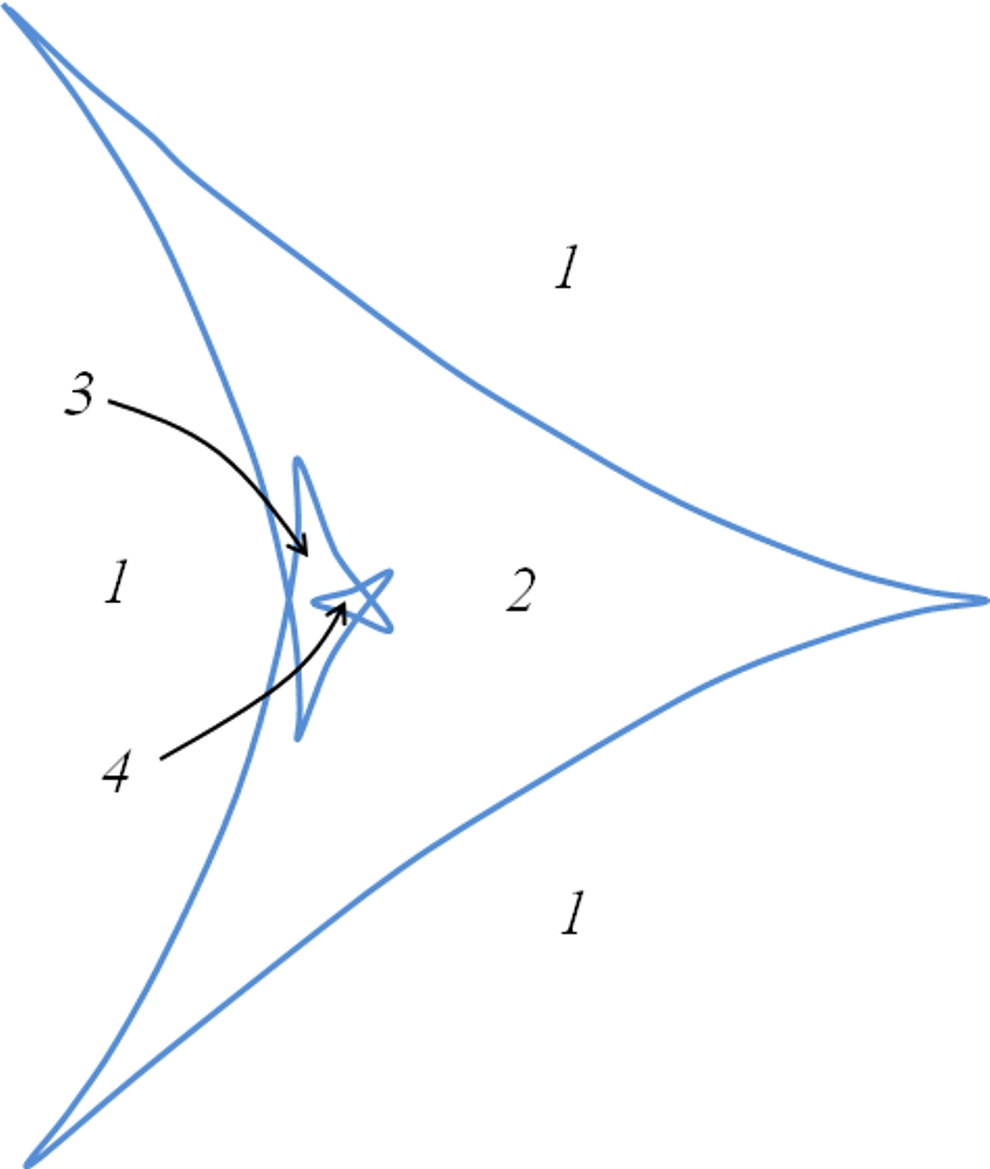}\includegraphics[width=0.25\textwidth]{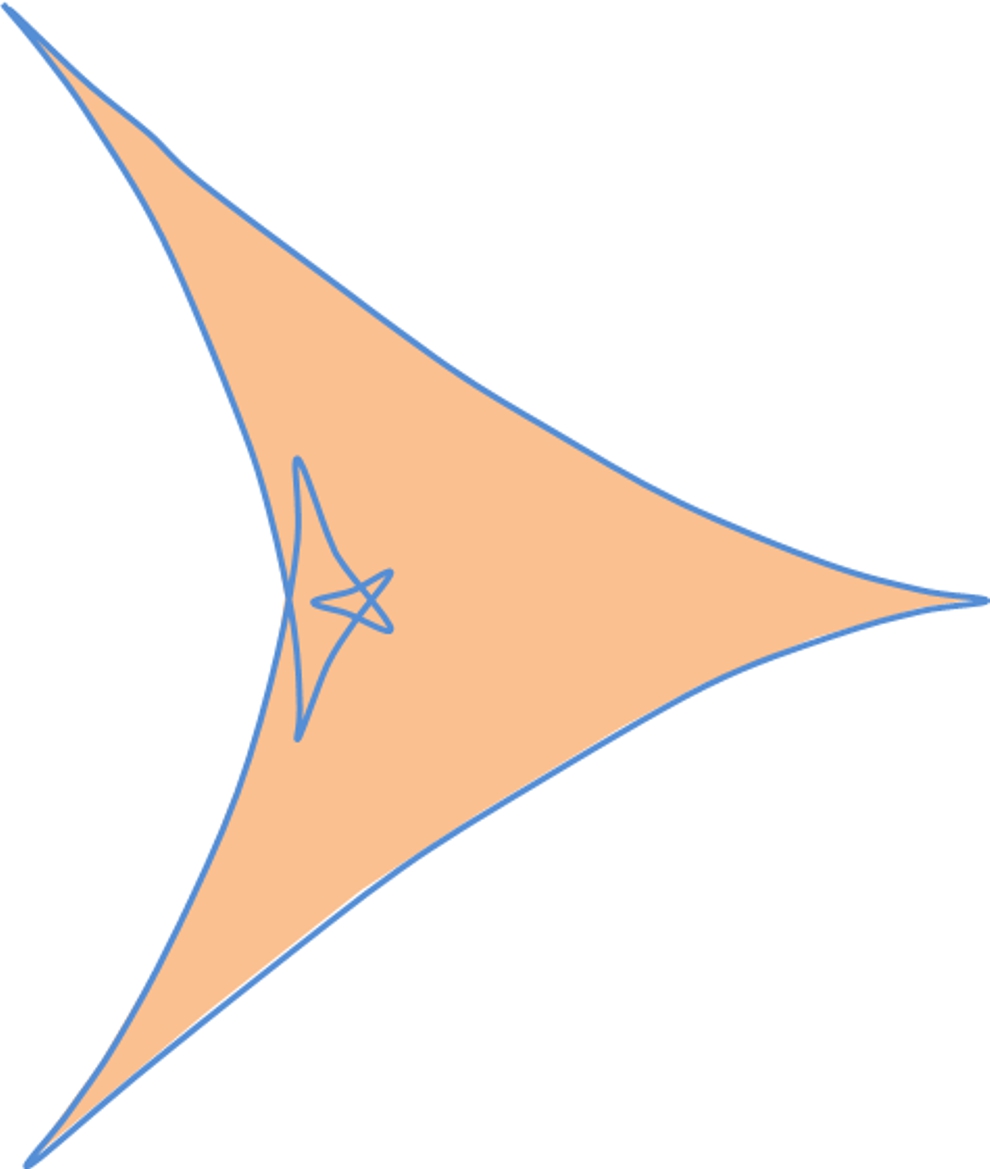}\includegraphics[width=0.25\textwidth]{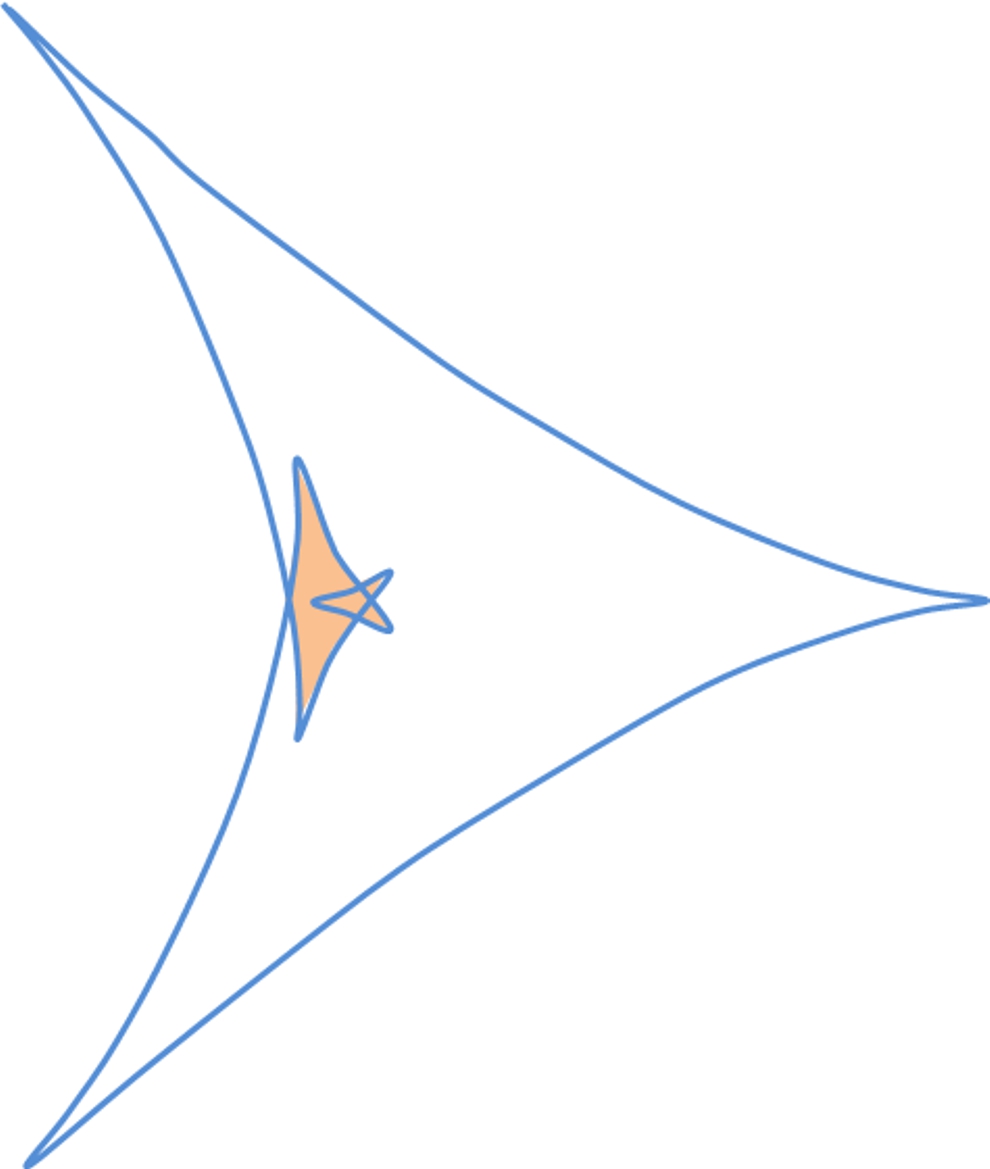}\includegraphics[width=0.25\textwidth]{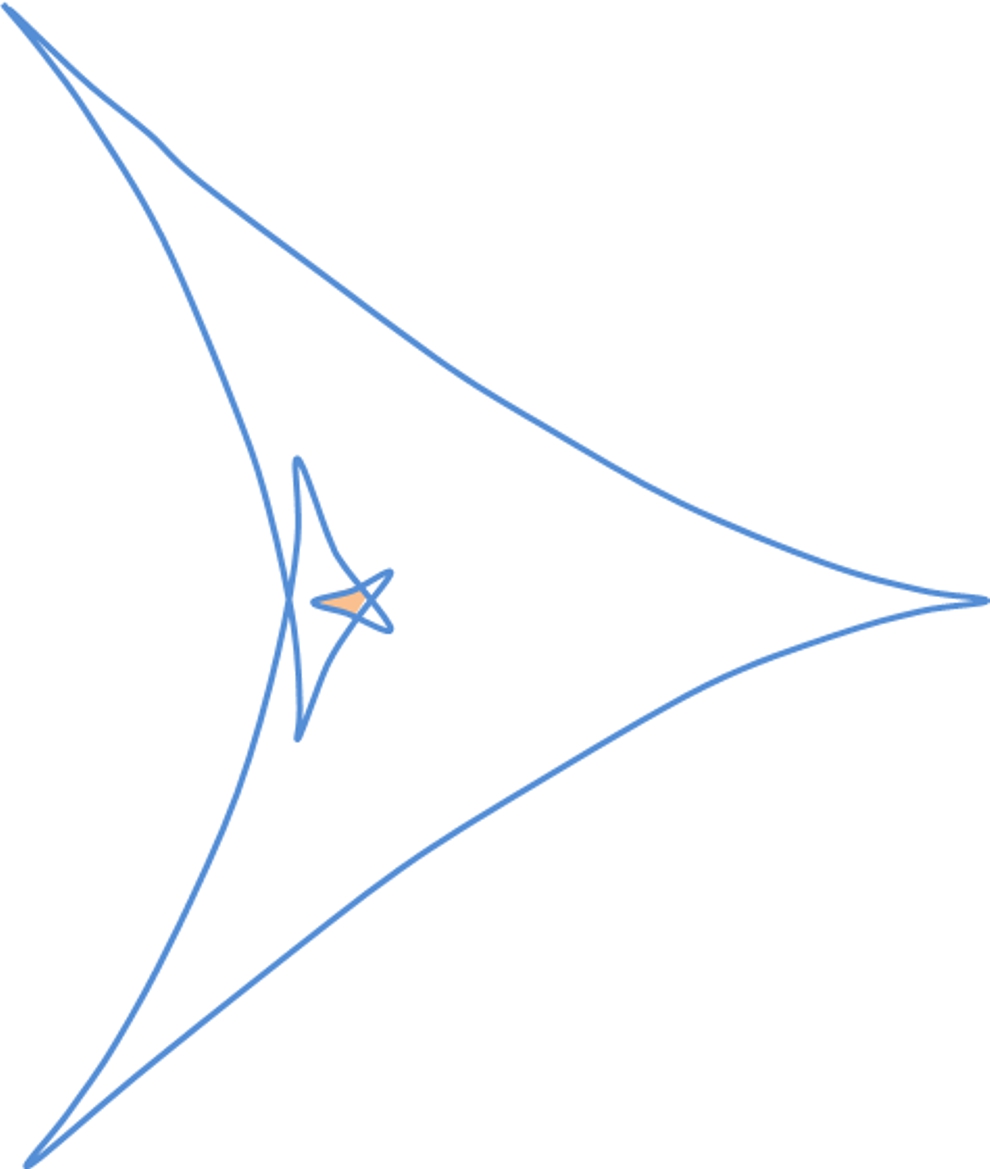}}\caption{On the far left a typical conjugate locus (taken from \cite{TWbif}) with the count marked in some regions; then shaded are $S_2, S_3$ and $S_4$, each of which are discs and have Euler charateristic 1. Thus the rotation index is 3, and the curve has 8 cusps as expected.}\label{figeul}
\end{figure}

The example in Figure \ref{figeul} is typical: the regions $S_m$ are usually discs or the disjoint union of discs, and hence have Euler characteristic $\geq 1$. If this were the case, the rotation index would be $\geq 1$ and hence $n$ could not be 2 via Theorem \ref{thmrot}.  Nonetheless, it is also possible that the $S_m$ are discs less discs, i.e. they could have `holes': 

\begin{lem}
For every hole in a region $S_m$, there is a disc.
\end{lem}
\begin{proof}
Suppose there is a region $S_m$ which has a hole; the count inside the hole must be $m-1$ (see the right of Figure \ref{figcount} for an example). The boundary of the hole must contain points of self-intersection of $\bbet$, as $\bbet$ is connected, and the interior is star-shaped with repsect to any point in the interior, due to local convexity of $\bbet$. If there is a point of self-intersection then there must be a sub-region of $S_m$ with count $m+1$. Either this region is a disc, in which case we are done, or it also has a hole. If it has a hole, then applying the same reasoning there must also be a region with count $m+2$, and so on. Eventually there must be a region without a hole, which is the disc we require.
\end{proof} 

\begin{thm} Let $\mS$ be a smooth strictly convex surface and let $p$ be a generic point in $\mS$. The conjugate locus of $p$ must have at least 4 cusps.
\end{thm}
\begin{proof}
According to Theorem \ref{thmrot}, we must rule out the case of $\bbet$ having rotation index 0 (if $p$ is generic then $R$ must have an even, and non-zero, number of stationary points which rules out $n=0,1,3$). But by Lemma 3 the sum of the Euler characteristics of the regions $S_m$ must be $\geq 1$, and hence the rotation index cannot be 0 and the number of cusps cannot be 2.
\end{proof}

\begin{figure}
{\includegraphics[width=0.45\textwidth]{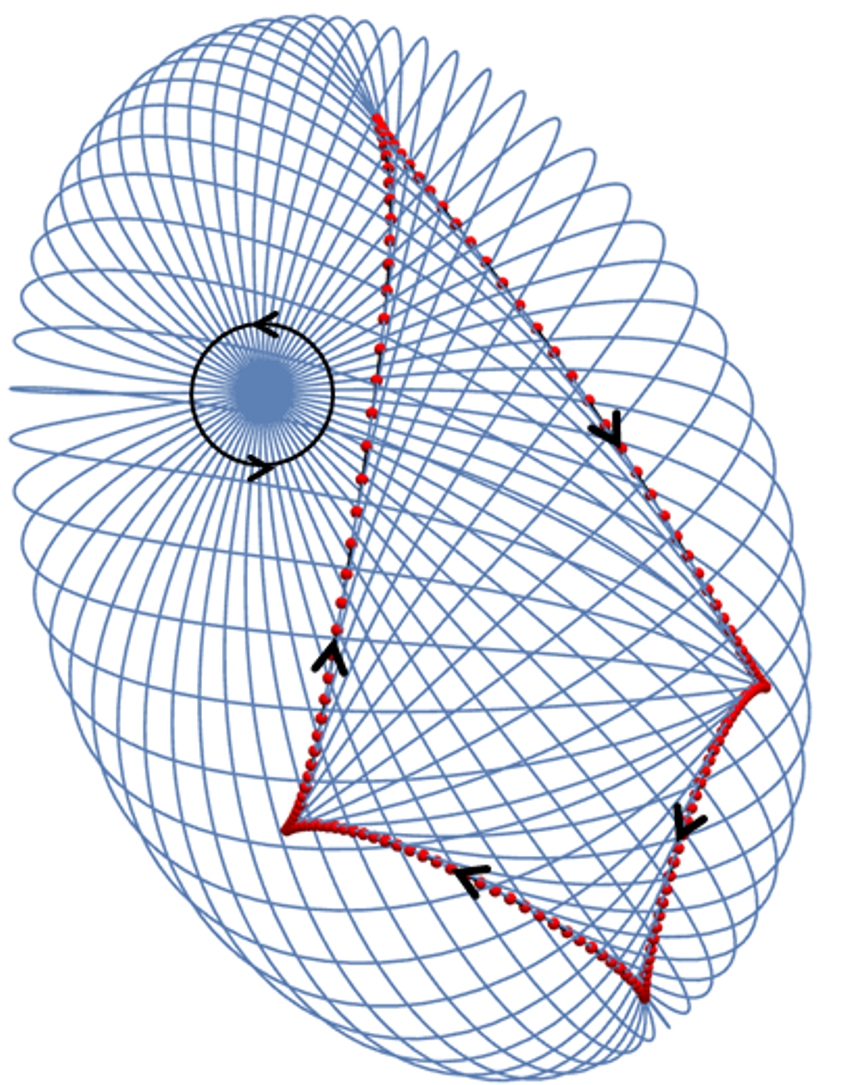}\hspace{1cm}\includegraphics[width=0.45\textwidth]{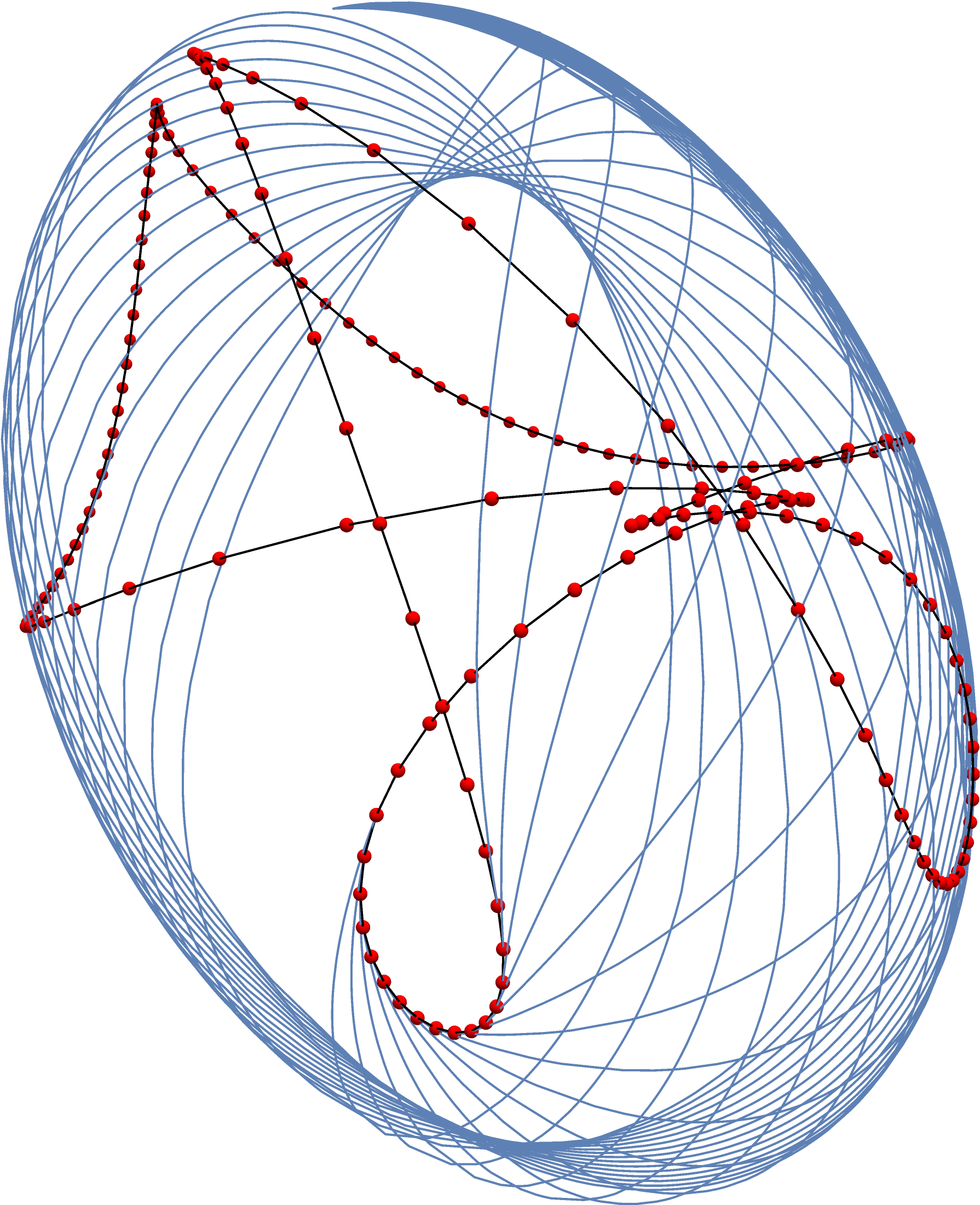}}\caption{Examples of $C_p^1$ (left) and $C_p^3$ (right) on the triaxial ellipsoid. Note the four cusps on the left as predicted by the Itoh Kiyohara theorem, and the smooth loop on the right. Blue lines are geodesics, black lines are the conjugate locus and the red dots are conjugate points.}\label{cps}
\end{figure}

\subsection{Geodesic curvature}

As it is relevant for this discussion, but seems to be missing from the literature, we derive an expression for the geodesic curvature of the conjugate locus. Recall from \eqref{tangcp} that the unit tangent vector to $\bbet(\psi)$ is $\bT(R(\psi),\psi)=\tfrac{\partial X}{\partial s}(R(\psi),\psi)$, so the $\psi$ derivative is \[ \frac{D}{d\psi}\bT(R(\psi),\psi)=\frac{D\bT}{\partial s}R'+\frac{D}{\partial \psi}\pfrac{X}{s}=\frac{D}{\partial s}\pfrac{X}{\psi}=\frac{D}{\partial s}\bmt{J} \] since $D\bT/\partial s=0$ by definiton of a geodesic. If $\bmt{J}=\xi\bN$ then $D\bmt{J}/\partial s=\xi_{,s}\bN$. Now letting $\sigma$ be the arc-length of $\bbet$ the Frenet equation for geodesic curvature \cite{kuhnel} gives \[ \frac{D}{d\sigma}\bT=\frac{d\psi}{d\sigma}\frac{D}{d\psi}\bT=\frac{1}{R'}\xi_{,s}\bN=k_g\bN \] and thus \be k_g(\psi)=\frac{\xi_{,s}(R(\psi),\psi)}{R'(\psi)}. \label{kgform} \ee This explains the appearance of the $\xi_{,s}$ term in \eqref{taylor}. We note that $\xi_{,s}(R(\psi),\psi)$ is never zero (and hence $k_g$ is never zero): since $R$ is defined by $\xi(R(\psi),\psi)=0$ then $\xi_{,s}(R(\psi),\psi)=0$ and the Jacobi equation \eqref{jaco} would imply $\xi\equiv 0$. Rather since $\xi_{,s}(0,\psi)=1$ and $s=R(\psi)$ is the {\it first} non-trivial zero of $\xi$ we have $\xi_{,s}(R(\psi),\psi)<0$. The formula applies equally well to $C_p^j$ (the locus of $j$th conjugate points to $p$) if we simply adapt $R$ to be the $j$th non-trivial zero of $\xi$ (note the sign of $\xi_{,s}$ will alternate from $C_p^j$ to $C_p^{j+1}$). Another expression for the geodesic curvature is found by differentiating $\bmt{J}(R(\psi),\psi)=0$ w.r.t. $\psi$ (see the `focal differential equation' in \cite{wolter1}) to give $k_g=-\xi_2(R(\psi),\psi)/(R')^2$ (see \cite{TWbif} for a definition of $\xi_2$). 

Note the sign of $k_g$ in \eqref{kgform} alternates between positive and negative depending on whether $R$ is decreasing or increasing on an arc of $C_p$, whereas using the orientation defined before Theorem \ref{thmrot} in Section 3.1 $k_g$ is strictly positive; this is because \eqref{kgform} is derived w.r.t. the tangents to radial geodesics which define an alternating orientation on $C_p$, similar to that described in Section 2.1. 

While the geodesic curvature in \eqref{kgform} diverges when $R$ is stationary (as we already knew), the total curvature \[ \int k_g d\sigma=\int\left(\frac{\xi_{,s}}{R'}\right)R'd\psi=\int\xi_{,s}(R(\psi),\psi)d\psi \] does not depend explicitly on $R$. On the sphere $\xi_{,s}(R(\psi),\psi)=-1$, so on a perturbed sphere we might expect the total curvature of an arc of the conjugate locus to be close to the difference in the values of $\psi$ at which $R$ is stationary which suggests an upper bound on $\int k_g d\sigma$, however moving away from the sphere it seems the arcs of the conjugate locus can become very curved indeed.

\section{Conclusions}

We cannot simply draw a plane curve with an even number of arcs and cusps and claim it to be the evolute of a plane oval; there are restrictions both geometrical (local convexity, sum of the alternating lengths vanishing) and topological (rotation index), as we have shown. In the same way we have shown that there are restrictions on the geometry and topology of the conjugate locus of a generic point on a smooth strictly convex surface, and these restrictions lead to structure which can be exploited. 

There are a number of generalizations of the results in this paper which are worth considering. We have already alluded to the `higher' conjugate locus $C_p^j$, and it is possible there is a formula generalizing \eqref{iandnorig}. Also we remind the reader that the conjugate locus of $p$ is the envelope of geodesics normal to the geodesic circle centred on $p$ (in this context the term `caustic' may be more appropriate). For small radius the geodesic circles are simple, and it is possible the caustics of simple curves on convex surfaces also satisfy \eqref{iandnorig} whereas nonsimple curves lead to \eqref{iandiandn} (possibly with the requirement the curves themselves have $k_g\neq 0$), however initial experiments suggest this context to be surprisingly complicated. Finally we have only considered convex surfaces however it is likely expressions such as \eqref{iandnorig} carry over to spheres with $K<0$ as long as the regions of negative Gauss curvature do not contain any closed geodesics (for example the spherical harmonic surfaces of \cite{TWspherical}, as opposed to the `dumbell' of \cite{Sinclair3}).

\bibliographystyle{plain}

\bibliography{clbib}

\appendix

\section{Additional diagrams}

\begin{figure}[h]
{\includegraphics[height=0.175\textheight]{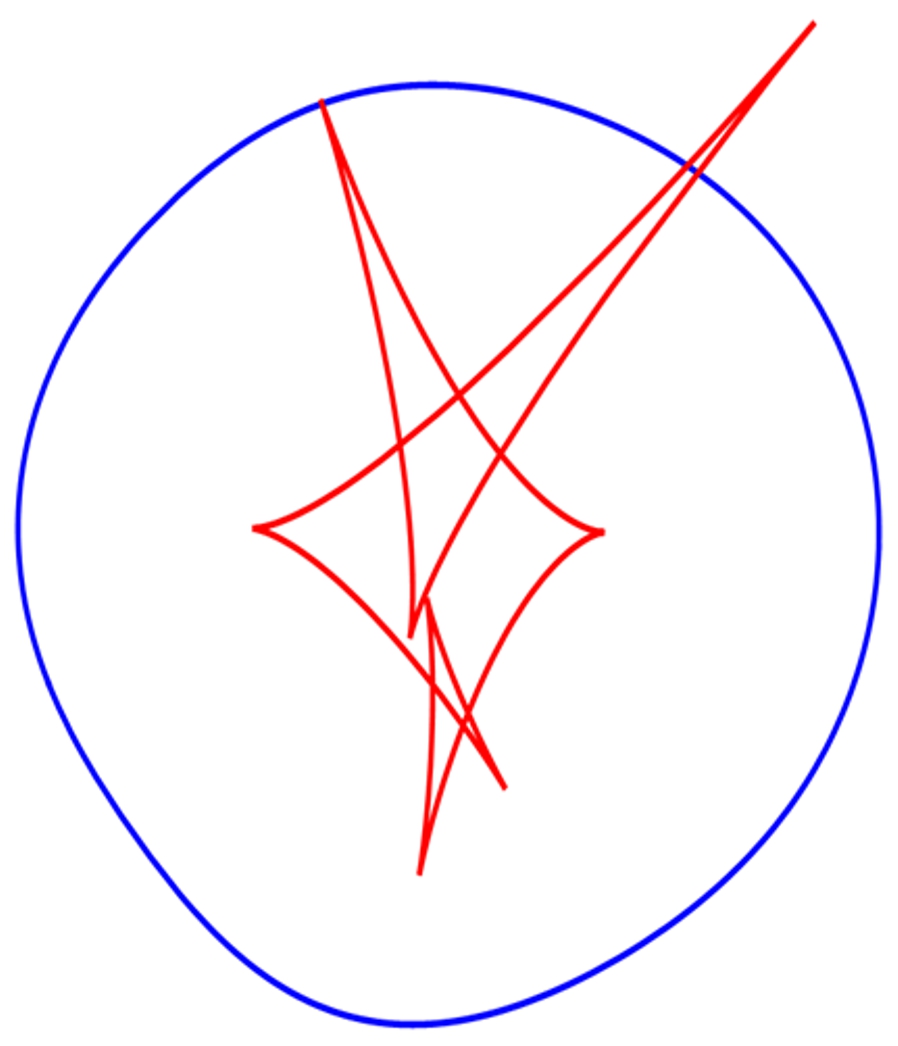}\hspace{0.2cm}\includegraphics[height=0.175\textheight]{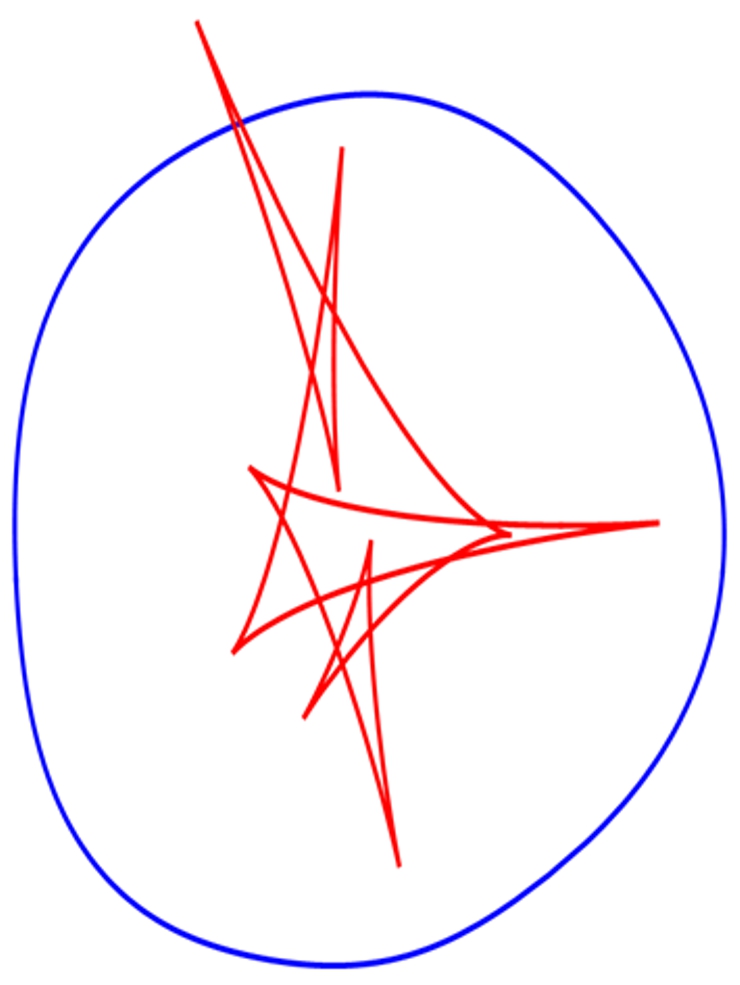}\hspace{0.2cm}\includegraphics[height=0.175\textheight]{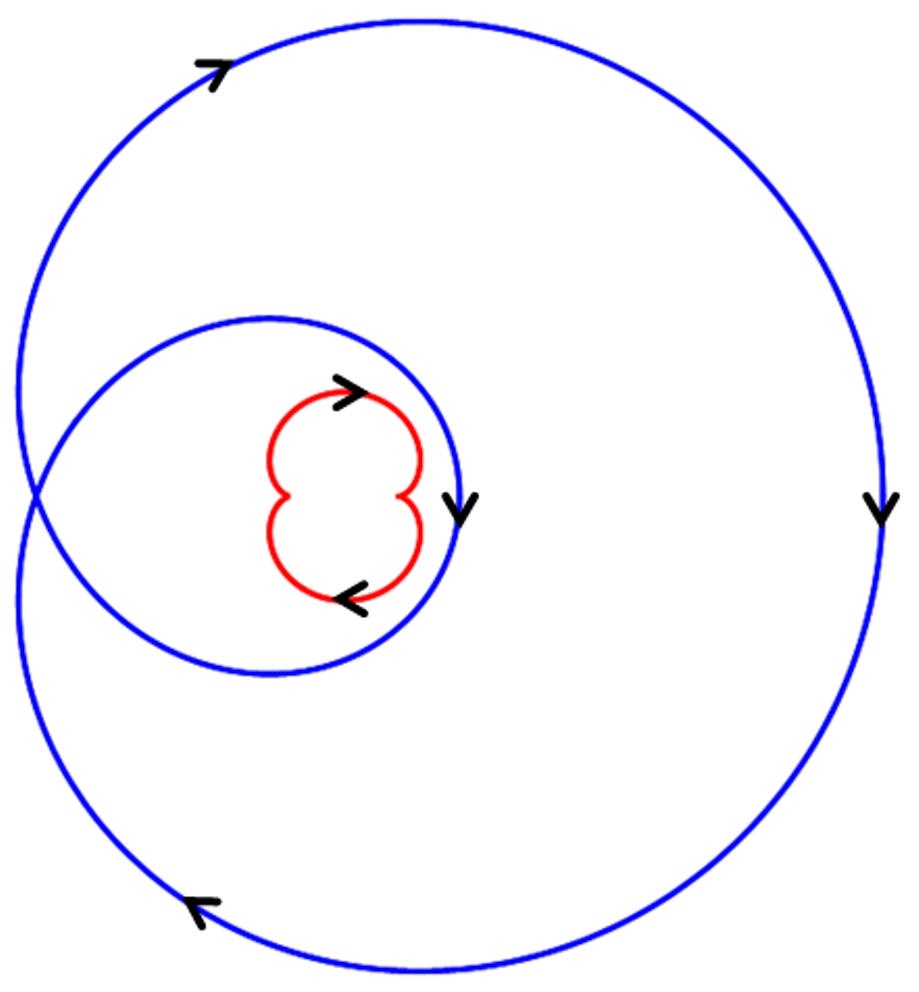}\\ \includegraphics[width=0.28\textwidth]{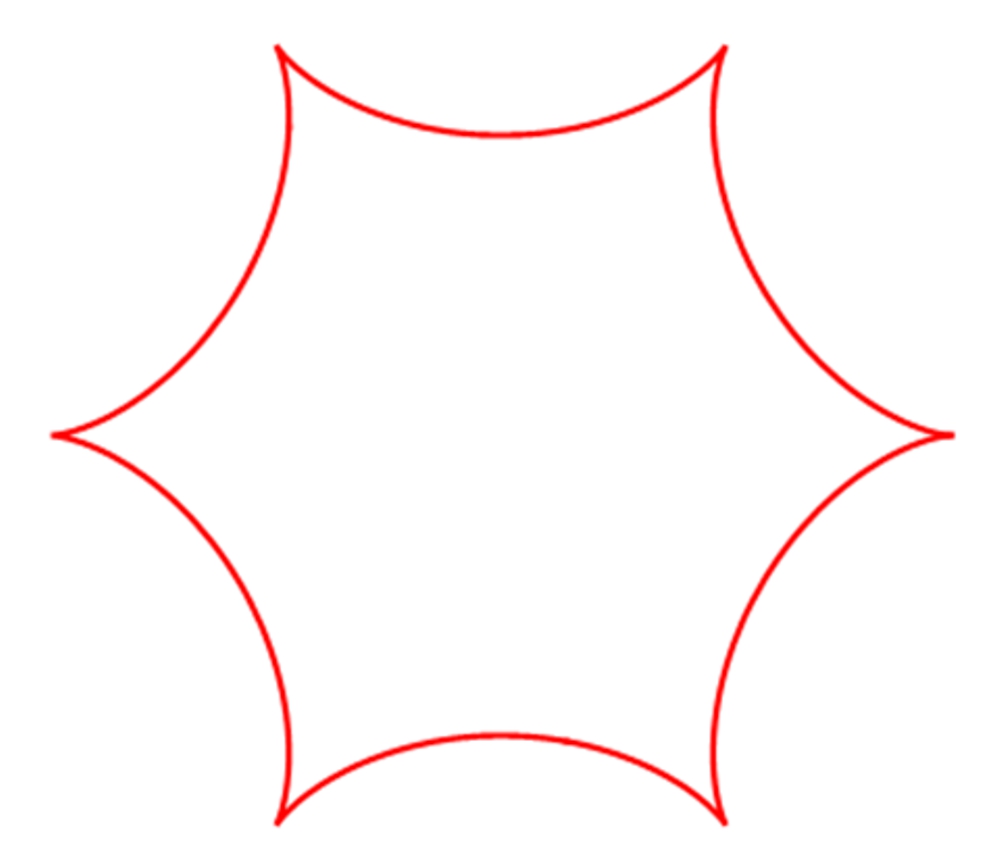}\hspace{0.2cm}\includegraphics[width=0.25\textwidth]{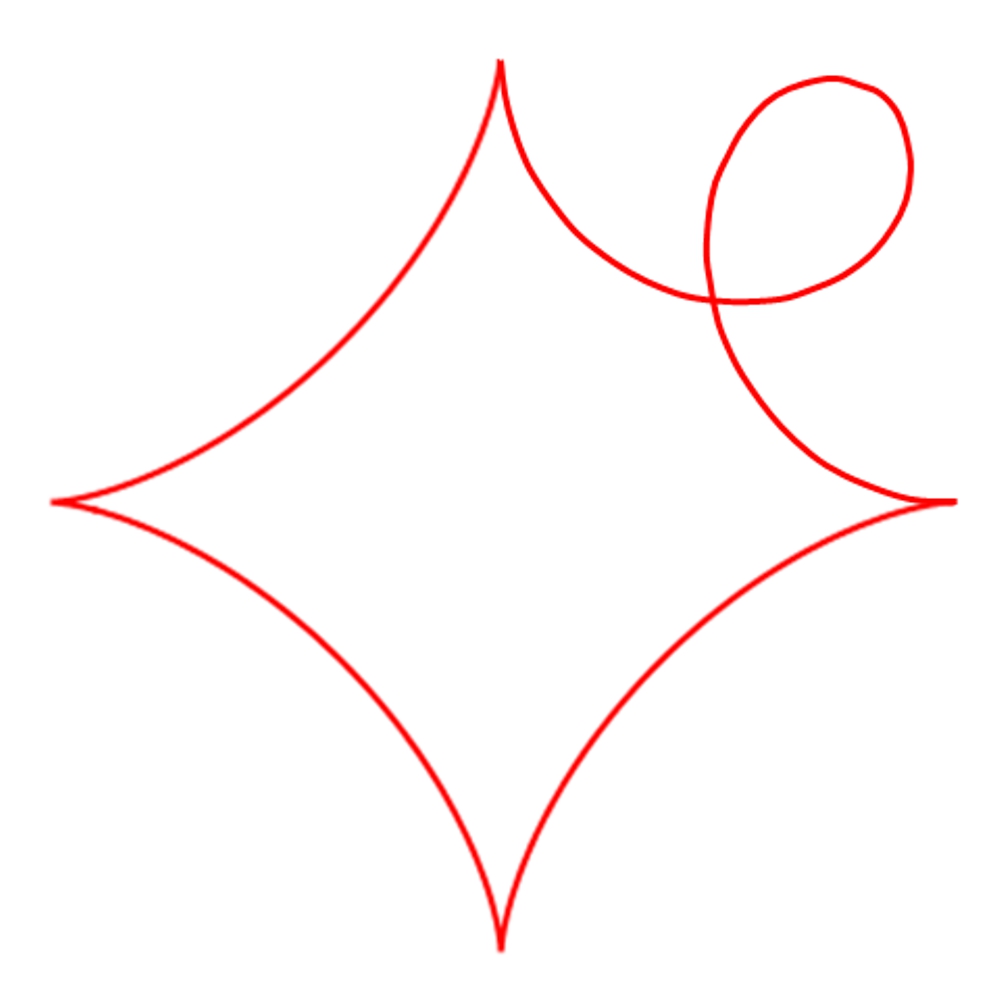}}\caption{Top row: examples of evolutes of plane curves. The left of Figure \ref{figevo}, together with the left and middle above, have $(i,n)=(2,6),(3,8),(4,10)$ respectively in accordance with \eqref{iandn}; the right has $(i,n,I)=(-1,2,-2)$ in accordance with \eqref{iandiandn}. Bottom row: these curves cannot be the conjugate locus of a generic point on a convex surface, or the evolute of a plane oval.}\label{appics}
\end{figure}

\section{Existence of smooth loops}

In the original version of this paper it was claimed that the evolute of a plane oval could not have a smooth loop, by which we mean an arc of the evolute between two consecutive cusps intersecting itself. To show this it was claimed that a $2\pi$-periodic function whose Fourier series begins at the second harmonic must have a zero in any interval of length $\pi$. Correspondence with Professors Bruna, Cufi and Reventos at the Universitat Aut\`{o}noma de Barcelona (March 2025) have shown these two claims to be false.

Their approach is to construct piecewise functions with the desired properties to act as support functions \cite{tabach}, and then fit Fourier series to them. For example the following \[ -7 \cos(2 t) + 10 \cos(4 t) - \frac{4096}{105\pi} \sin(3 t) + \frac{4096 \sin(5 t)}{315 \pi} \] is a $2\pi$-periodic function whose Fourier series starts at the second harmonic, but all four zeroes are in the interval $(0,\pi)$ and therefore there are no zeroes in an interval larger than $\pi$, thus disproving the second claim mentioned above.

Moreover if we take the piecewise function \[ \left\{ \begin{array}{ll} t+\tfrac{3}{4}\left(1-\sin\left(\tfrac{4}{3}t\right)\right)+40 & 0<t<\tfrac{3\pi}{2} \\ -3t+\tfrac{109}{108}\sin(4t)-\tfrac{7}{54}\sin(8t)+6\pi+\tfrac{3}{4}+40 & \tfrac{3\pi}{2}<t<2\pi \end{array}\right. \] and fit a Fourier series to it (of about 20 terms), then use this as the support function to generate a smooth plane oval and its associated evolute, we get the curves in Figure \ref{figloop}, thus disproving the first claim.

\begin{figure}[h]
{\includegraphics[width=0.4\textwidth]{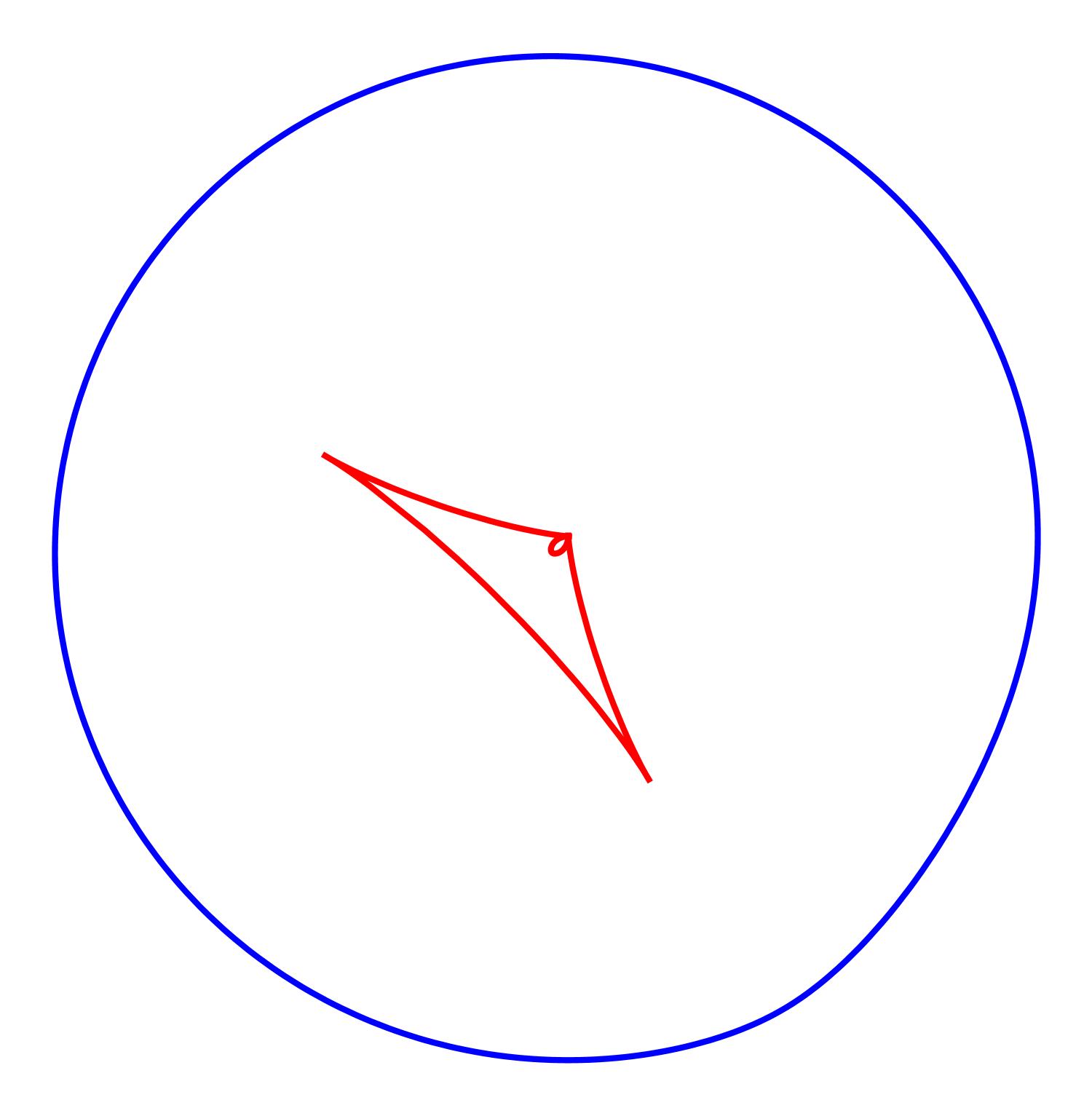}\hspace{0.3cm}\includegraphics[width=0.4\textwidth]{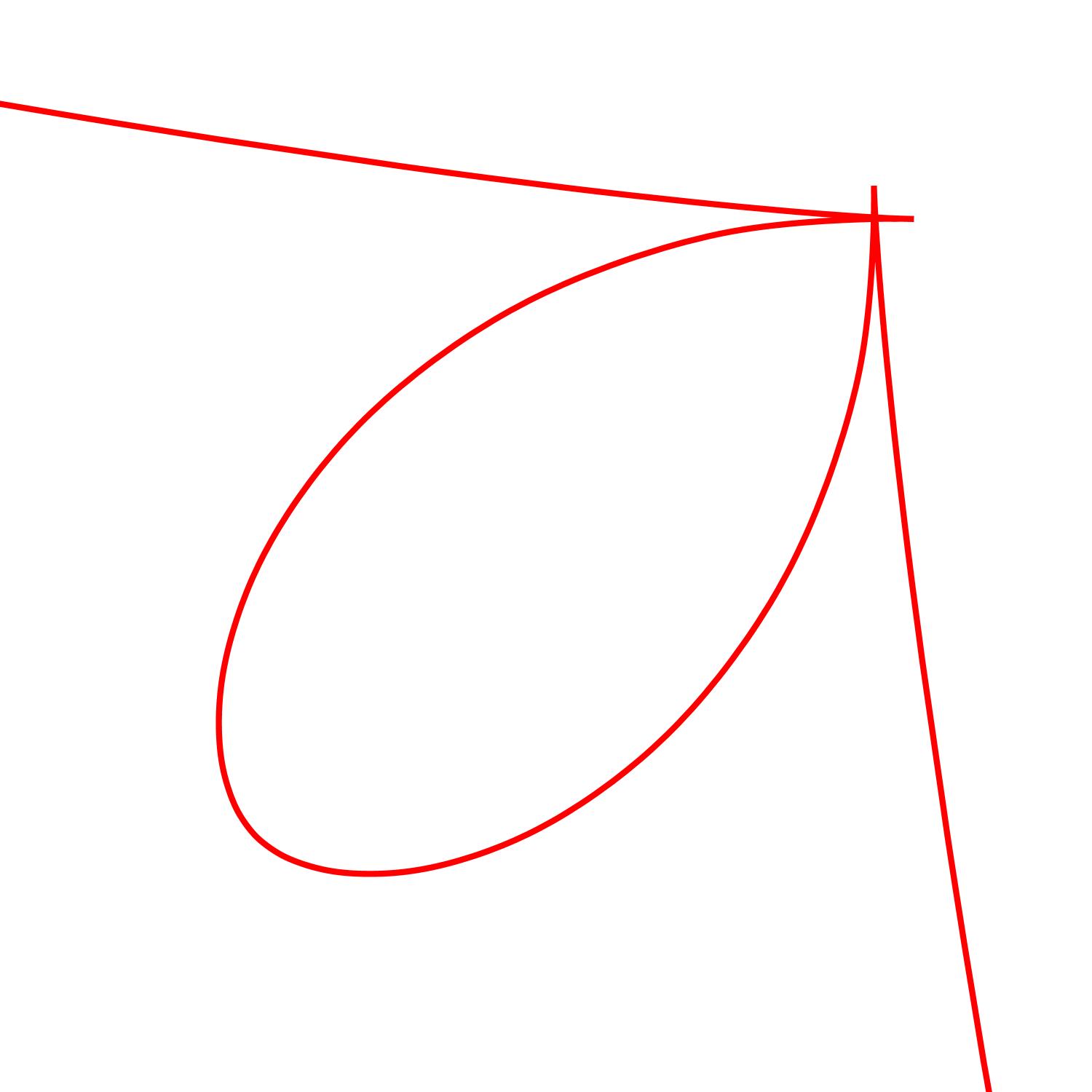}\hspace{0.3cm}}\caption{On the left a smooth ($C^\infty$) closed convex plane curve and its evolute; on the right a close-up showing the evolute with a smooth loop.}\label{figloop}
\end{figure}

\end{document}